\documentclass{cimart}

\newenvironment{eq}{\begin{equation}}{\end{equation}}
\renewcommand{\Ref}[1]{(\ref{#1})}

\newcommand{\Char}{\mathop{\rm char}}

\newcommand{\FF}{\mathbb{F}}

\newcommand{\RR}{\mathbb{R}}

\newcommand{\ZZ}{\mathbb{Z}}

\newcommand{\NN}{\mathbb{N}}

\newcommand{\si}{\sigma}
\newcommand{\al}{\alpha}
\newcommand{\be}{\beta}
\newcommand{\ga}{\gamma}

\newcommand{\de}{\delta}
\newcommand{\De}{\Delta}

\newcommand{\Sym}{{\mathcal S}}

\newcommand{\fr}{\mathop{\rm frac}}       

\title{
    Separating symmetric polynomials over finite fields
    }

\author{
    Artem Lopatin, Pedro Antonio Muniz Martins and Lael Viana Lima
    }

\authorinfo[Artem Lopatin]{State University of Campinas, Brazil}{dr.artem.lopatin@gmail.com}

\authorinfo[Pedro Antonio Muniz Martins]{State University of Campinas, Brazil}{p242894@dac.unicamp.br}

\authorinfo[Lael Viana Lima]{State University of Campinas, Brazil}{l176809@dac.unicamp.br}

\abstract{The set $S(n)$ of all elementary symmetric polynomials in $n$ variables is a minimal generating set for the algebra of symmetric polynomials in $n$ variables, but over a finite field $\FF_q$ the set $S(n)$ is not a minimal separating set for symmetric polynomials in general. We determine when $S(n)$ is a minimal separating set for the algebra of symmetric polynomials having the least possible number of elements.
    }

\keywords{Finite field, Symmetric group, Symmetric polynomials, Invariant theory, Separating invariants.}

\msc{13A50 (primary); 12E20 (secondary).
    }

\begin{document}

\VOLUME{33}
\YEAR{2025}
\ISSUE{1}
\NUMBER{5}
\DOI{https://doi.org/10.46298/cm.14627}
\section{Introduction}\label{section_intro}

\subsection{Symmetric polynomials} Assume that $\FF$ is an arbitrary field (finite or infinite) and denote by $\FF_q$ the finite field of order $q$ with the characteristic $p=\Char{\FF_q}$. 

Consider an $n$-dimensional vector space $V$ over a field $\FF$ with a fixed basis, where $n\geq2$. For $v \in V$ let $v_i$ denote the $i^{\rm th}$-coordinate with respect to this basis of $V$. The symmetric group $\Sym_n$ acts on $V$ by permutations of the coordinates with respect to the fixed basis of $V$. Namely, for $\si\in \Sym_n$ and $v=(v_1,\ldots,v_n)\in V$ we have $\si\cdot v=(v_{\si^{-1}(1)},\ldots,v_{\si^{-1}(n)})$. The coordinate ring $\FF[V]  = \FF[ x_1, \ldots, x_n]$ of $V$ is isomorphic to the symmetric algebra $S(V^{\ast})$ over the dual space $V^{\ast}$ with the dual basis $x_1, \ldots ,x_n$ to the fixed basis of $V$. The group $\Sym_n$ acts on the set $\{x_1, \ldots ,x_n\}$ by $\si\cdot x_i=x_{\si(i)}$ and this action is extended to the action of $\Sym_n$ on $\FF[V]$. The algebra of {\it $\Sym_n$-invariants} \[ \FF[V]^{\Sym_n}=\{f\in \FF[V] \,|\, \si\cdot f=f \text{ for all } \si\in\Sym_n\} \] is the algebra of symmetric polynomials. It is well known that the algebra $\FF[V]^{\Sym_n}$ is minimally (with respect to inclusion) generated by the set 
$$S(n)=\{s_t(x_1,\ldots,x_n)\,|\,1\leq t\leq n\}$$
of all elementary symmetric polynomials $s_t(x_1,\ldots,x_n)=\sum_{1\leq i_1< \cdots < i_t\leq n} x_{i_1}\cdots x_{i_t}$.

Any element $f$ of $\FF[V]$ can be considered as a function $f:V\to\FF$. Obviously, any $f\in \FF[V]^{\Sym_n}$ has a constant value over every $\Sym_n$-orbit on $V$. Given a subset $S$ of $\FF[V]^{\Sym_n}$, we say that elements $u,v$ of $V$ {\it are separated by $S$} if there exists an invariant $f\in S$ with \mbox{$f(u)\neq f(v)$}.   If  $u,v\in V$ are separated by $\FF[V]^{\Sym_n}$, then we simply say that they {\it are separated}. A subset $S\subset \FF[V]^{\Sym_n}$ is called {\it separating} if for any $u, v$ from $V$ that are separated we have that they are separated by $S$. We say that a separating set is {\it minimal} if it is minimal with respect to inclusion. Obviously, any generating set is also separating.  Minimal separating sets for different actions of groups were constructed in~\cite{Cavalcante_Lopatin_2020, Domokos_2020, Domokos_2020_add, Lopatin_Ferreira_2023, Kaygorodov_Lopatin_Popov_2018, Kemper_Lopatin_Reimers_2022, Lopatin_Reimers_2021, Lopatin_Zubkov_ANT, Lopatin_2024_IJAC, Lopatin_Martins_2023, Reimers_2020}. 

In the case of an algebraically closed field $\FF$ as well as in the case of $\FF=\RR$ the set $S(n)$ is a minimal separating set for $\FF[V]^{\Sym_n}$ having the least possible number of elements. On the other hand, over a finite field a minimal separating set for the algebra of symmetric polynomials is not known in general. For every $n\in \NN$ denote
$$[n]_q = \{j p^k \,|\,1\leq j< q,\; k\in \NN_0,\; j p^k \leq n \} \text{ and }$$
$$S_q(n) = \{s_t(x_1,\ldots,x_n)\,|\, t\in [n]_q\},$$
where $\NN_0=\NN\sqcup \{0\}$. 
In 1964 Aberth~\cite{Aberth_1964} established that $S_p(n)$ is a separating set for $\FF_p[V]^{\Sym_n}$ for a prime $p$. In~\cite{Kemper_Lopatin_Reimers_2022} it was proven that the set $S_2(n)$ is a minimal separating set for $\FF_2[V]^{\Sym_n}$ having the least possible number of elements.  Recently, Domokos and Mikl\'osi~\cite{Domokos_Miklosi_2023} extended the result of Aberth to the case of an arbitrary finite field. Namely, they proved that $S_q(n)$ is a separating set for $\FF_q[V]^{\Sym_n}$. Nevertheless, the set $S(n)$ is a minimal separating set for $\FF_q[V]^{\Sym_n}$ in some cases. 

\subsection{Results}
In Theorem~\ref{theo_min_sep} and Corollary~\ref{corArt} we prove that $S(n)$ is a minimal separating set for $\FF_q[V]^{\Sym_n}$ having the least possible number of elements if and only if $n\leq \chi_q$, where $\chi_q$ is defined by formula~(\ref{eq_zeta}). The explicit values of $\chi_q$ for $q\leq 10^4$ are given in Remark~\ref{rem1}. Since $\chi_{q} \ge \lfloor \ln (\ln q) \rfloor$ by Theorem~\ref{theo2}, for every $n\geq2$ there exists $q$ such that $S(n)$ is a minimal separating set for $\FF_q[V]^{\Sym_n}$ having the least possible number of elements (see Corollary~\ref{cor3}). In Proposition~\ref{prop_F3} we determine when the separating set $S_3(n)$ for $\FF_3[V]^{\Sym_n}$ has the least possible number of elements.

\subsection{Auxiliaries}
Since the number of $\Sym_n$-orbits on $V$ is the binomial coefficient $\binom{n+q-1}{q-1}$, Theorem 1.1 of~\cite{Kemper_Lopatin_Reimers_2022} implies that the least possible number of elements of a separating set for $\FF_q[V]^{\Sym_n}$ is 
\begin{eq}\label{eq_ga}
\ga=\ga_q(n)= \Big\lceil \log_{q} \frac{(n+q-1)\cdot \, \ldots \, \cdot (n+1)}{(q-1)!} \Big\rceil
\end{eq}

Consider some properties of the floor and the ceiling functions. Obviously, for $x\in \RR$ and $n \in \ZZ$ we have  \[ \lfloor x+n \rfloor= \lfloor x \rfloor +n, \lceil x+n \rceil = \lceil x \rceil +n , \fr(x+n) = \fr(x) \text{, and }-\lfloor x \rfloor = \lceil -x \rceil, \] where $\fr(x)$ stands for the fractional part of $x$, i.e., $x=\lfloor x \rfloor+\fr(x)$.

\begin{remark}\label{remark1} For $a,b\in\RR$ with $b \not\in \ZZ$ we have 

\begin{enumerate}
\item[(a)]$$\lfloor 2\, a \rfloor = \left\{
\begin{array}{cc}
2 \lfloor a \rfloor + 1 & \text{ if }\quad \fr(a)\geq 1/2\\\
2 \lfloor a \rfloor &  \text{ if }\quad \fr(a)< 1/2\\
\end{array}
\right.;$$

\item[(b)] $$\lfloor a-b \rfloor  = \left\{
\begin{array}{cc}
 \lfloor a \rfloor+ \lfloor -b \rfloor +1 & \text{ if }\quad \fr(a)\geq\fr(b)\\
\lfloor a \rfloor+ \lfloor -b \rfloor    &  \text{ if }\quad \fr(a)<\fr(b)\\
\end{array}
\right..$$
\end{enumerate}
\end{remark}

\section{The case of \texorpdfstring{$\FF_3$}{F3}}\label{section_F3}

For short, we denote 
$$a_{r}=3^{\frac{r}{2}} \text{ and } b_{r}=\frac{-3+\sqrt{8\cdot 3^r +1}}{2}$$ 
for $r\in\NN_0$. Note that 
\begin{eq}\label{claim3}
a_{r}<b_{r}<a_{r+1} \;\; \text{for all} \;\; r\geq 3.
\end{eq}%

\begin{lemma}\label{lemma1}
For every $n \geq 1$ we have
$$2 \lfloor \log_{3}n \rfloor= \lfloor 2\, \log_{3} n \rfloor + \al,$$ 
where
\begin{enumerate}
\item[$\bullet$] $\al=0$, if $n \in [a_{2r}, a_{2r+1})$ for some $r\in\NN_0;$
    
\item[$\bullet$] $\al=-1$, if $n \in [a_{2r+1}, a_{2r+2})$ for some $r\in\NN_0$.
\end{enumerate}
\end{lemma}
\begin{proof} 
By part (a) of Remark~\ref{remark1}, the statement of the lemma follows from the following claim:
\begin{eq}\label{claim1}
\fr( \log_{3} n )<\frac{1}{2} \;\;\text{ if and only if }\;\; n \in [a_{2r}, a_{2r+1}) \text{ for some } r\in\NN_0. 
\end{eq}

Note  $\fr(\log_{3}n)=0$ if and only if $n=a_{2r}$ for some $r\in \NN_0$. Since $\log_{3}n$ is a strictly increasing function, then $\fr(\log_{3}n)$ is also strictly increasing on every interval $[a_{2r}, a_{2r+2} )$ with $r\in\NN_0$. 
The equality $\fr(\log_{3} a_{2r+1})=1/2$  for every $r\in\NN_0$ concludes the proof of claim~(\ref{claim1}).
\end{proof}

\begin{lemma}\label{Lemma2}
Assume that $n \geq6$. Then for  $f_{1}(x)=\log_{3}x^2$, $f_{2}(x)=\log_{3}\frac{(x+1)(x+2)}{2}$, and $f_{3}(x)=\log_{3} \cfrac{2}{1+\frac{3}{x}+\frac{2}{x^2}}$ we have
$$ 
\lfloor f_{1}(n) \rfloor + \lfloor -f_{2}(n) \rfloor =\lfloor f_{3}(n) \rfloor + \be,$$
where
\begin{enumerate}
\item[$\bullet$] $\be=0$, if $n \in [a_{r}, b_{r})$ for some $r \in \NN;$ 
 
\item[$\bullet$] $\be=-1$, if $n \in [b_{r}, a_{r+1})$ for some $r \in \NN$.
\end{enumerate}
\end{lemma}
\begin{proof} Since $n\geq6$, we have $a_3<n$. Hence, $a_{r}<b_{r}<a_{r+1}$ in case $n \in [a_{r}, a_{r+1})$ by~(\ref{claim3}). It is easy to see that $f_2(n)\not\in\ZZ$, since in case $(n+1)(n+2)=2\cdot 3^k$ for some $k\in\NN$ we obtain a contradiction. 

We assume that $x$ lies in $\RR_{+}=(0,+\infty)$. Since $f_1(x)-f_2(x)=f_3(x)$ and $f_2(n)\not\in\ZZ$, part (b) of Remark~\ref{remark1} implies that the statement of the lemma follows from the next claims:
\begin{eq}\label{claim2a}
\fr( f_{1}(x) )< \fr( f_{2}(x)), \; \text{ if }\;\; 
x \in [a_r, b_r) \text{ for some } r\geq3,
\end{eq}%
\begin{eq}\label{claim2b}
\fr( f_{1}(x) )\ge \fr( f_{2}(x)), \text{ if }
x\in [b_{r}, a_{r+1}) \text{ for some } r\geq3.
\end{eq}%
We have $\fr(f_{1}(x))=0$ if and only if $x=a_r$ for some $r \in \NN_0$. Similarly, $\fr(f_{2}(x))=0$ if and only if $x = b_r$ for some $r \in \NN_0$. Since $f_{1}(x)$ and $f_{2}(x)$ are strictly increasing, then $\fr(f_{1}(x))$ and $\fr(f_{2}(x))$ are also  strictly increasing on intervals $[a_{r},a_{r+1})$ and $[b_{r},b_{r+1})$, respectively, where $r \in \NN_0$.

Since $f'_{1}(x)>f'_{2}(x)$ for all $x\in\RR_{+}$ and $\fr(f_1(b_r))>\fr(f_2(b_r))$, we obtain claim~(\ref{claim2b}). 

Assume that $\fr(f_{1}(x)) \ge \fr(f_{2}(x))$ for some $x \in [a_{r},b_{r})$ with $r\geq3$. Then there exists $x_0 \in [a_{r},b_{r})$ with $\fr(f_{1}(x_0)) = \fr(f_{2}(x_0))$. Since $f_{1}$ increases faster than $f_{2}$, we have $\fr(f_{1}(x)) > \fr(f_{2}(x))$ for all $x\in[x_0,b_r)$. Then the equality $\lim\limits_{x\to b_r^{-}} \fr(f_2(x))=1$ implies that $\lim\limits_{x\to b_r^{-}} \fr(f_1(x))=1$, i.e., $\fr(f_1(b_r))=0;$ a contradiction to inequalities~(\ref{claim3}).  Hence claim~(\ref{claim2a}) is proven.
\end{proof}

\begin{proposition}
\label{prop_F3} Let $\De= \#S_3(n) - \gamma_3(n)$ be the difference between the number of elements of the separating set $S_3(n)$ for $\FF_3[V]^{\Sym_n}$ and the least possible number of elements of a separating set for  $\FF_3[V]^{\Sym_n}$.
Then

\begin{enumerate}
    \item[$\bullet$] $\De=0$ in case $2\leq n\leq 8;$ 

    \item[$\bullet$] \( \begin{aligned}
        & \De= \left\{ 
       \begin{array}{cl}
       0, & \text{ if } n \in  [b_{2r},2\,a_{2r}) \cup [b_{2r+1}, a_{2r+2})  \text{ for some }r\in\NN \\
       1, & \text{ otherwise } \\
       \end{array}
       \right.\\
     & \text{ in case $n\geq 9$.}
     \end{aligned}
     \)
\end{enumerate}
\end{proposition}
\begin{proof} It is easy to see that $\#S_3(n) = 2\, \lfloor \log_{3}n \rfloor + \de$, where 
\begin{enumerate}
\item[$\bullet$] $\de=1$ in case $n\in [a_{2r},2\, a_{2r})$ for some $r\in\NN_0;$

\item[$\bullet$] $\de=2$ in case  $n \in [2\, a_{2r}, a_{2r+2})$  for some $r\in\NN_0$.
\end{enumerate}
Since $\ga_3(n) = \Big\lceil \log_{3}\frac{(n+2)(n+1)}{2} \Big\rceil$ by formula~(\ref{eq_ga}), we obtain
$$\De =2 \lfloor \log_{3}n \rfloor + \de -\Big\lceil \log_{3}\frac{(n+2)(n+1)}{2} \Big\rceil. $$%

For $2\leq n\leq 8$ by straightforward calculations, we can see that $\De=0$.

Assume $n\geq 9$. Then $a_4\leq n$ and inequalities~(\ref{claim3}) imply that
$$a_{2r}<b_{2r}<a_{2r+1}<2\, a_{2r}<b_{2r+1}< a_{2r+2}$$
in case $n \in [a_{2r}, a_{2r+2})$ for some $r\in\NN$. Note that here we have $r\geq2$.

Using the properties of ceiling functions and  Lemma \ref{lemma1}, we obtain
$$\De = \lfloor 2 \log_{3}n \rfloor + \Big\lfloor -\log_{3}\frac{(n+2)(n+1)}{2} \Big\rfloor+ \alpha + \de.$$%
Hence, Lemma \Ref{Lemma2} together with the fact that $\Big\lfloor \log_{3} \frac{2}{1+\frac{3}{n}+\frac{2}{n^2}} \Big\rfloor=0$ in case $n\geq4$ implies  
$$\De =  \al + \be+ \de,$$
where $\al$ and $\be$ are the same as in Lemmas \ref{lemma1} and \ref{Lemma2}, respectively.  
We complete the proof case-by-case consideration. Namely,
\begin{enumerate}
    \item[$\bullet$]  for $n\in [a_{2r},b_{2r})$ we have $\al + \be+ \de = 0+0+1=1;$
    \item[$\bullet$]  for $n\in [b_{2r},a_{2r+1} )$ we have $\al + \be+ \de = 0-1+1=0;$
    \item[$\bullet$]  for $n\in [a_{2r+1}, 2\, a_{2r})$ we have $\al + \be+ \de = -1+0+1=0;$
    \item[$\bullet$]  for $n\in [2\, a_{2r}, b_{2r+1})$ we have $\al + \be+ \de = -1+0+2=1;$
    \item[$\bullet$]  for $n\in [b_{2r+1}, a_{2r+2})$ we have $\al + \be+ \de = -1-1+2=0$.
\end{enumerate}
\end{proof}

\section{The general case}\label{section_general}

\begin{theorem}\label{theo_min_sep} 
The set $S(n)$ is a minimal separating set for $\FF_q[V]^{\Sym_n}$ having the least possible number of elements if and only if $n< x_0 $, where $x_0=x_0(q)\in\RR_{\geq 1}$ is the unique solution of the following equation
$$
q^{x-1} = (x+1)\Big(\frac{x}{2}+1\Big)\cdot \ldots \cdot \Big(\frac{x}{q-1}+1\Big)
$$
over $\RR_{\geq 1}=[1,+\infty)$. 
Moreover, 
\begin{enumerate}
\item[$\bullet$] $x_0>1;$

\item[$\bullet$] $x_0 < q$ in case $q>3$.
\end{enumerate}
\end{theorem}
\begin{proof}
Since $\#S(n)= n$ and $\ga = \ga_q(n) = \Big\lceil \log_{q} \frac{(n+q-1)\cdot \, \ldots \, \cdot (n+1)}{(q-1)!} \Big\rceil$ is the least possible number of elements of a separating set for $\FF_q[V]^{\Sym_n}$ by formula~(\ref{eq_ga}), using the properties of the floor and ceiling functions we obtain
$$
\#S(n)-\gamma = \Bigg\lfloor \log_{q} \frac{(q-1)!\cdot q^n}{(n+q-1)\cdot \, \ldots \, \cdot (n+1)}  \Bigg\rfloor.
$$%
Hence,

\begin{eq}\label{eq_new11}
\#S(n)=\gamma \; \text{ if and only if } \; \frac{(q-1)!\cdot q^n}{(n+q-1)\cdot \, \ldots \, \cdot (n+1)}<q.
\end{eq}

\noindent{}Therefore,  
$$\#S(n)=\gamma   \;\text {if and only if }\; q^{n-1} < (n+1)\Big(\frac{n}{2}+1\Big)\cdot \ldots \cdot \Big(\frac{n}{q-1}+1\Big). $$

\noindent{}Applying $\ln$ to both sides, we obtain that 
$$\#S(n)=\gamma   \;\text { if and only if }\; f_1(n)<f_2(n), $$%
where $f_1(x)=(x-1)\ln q$ and $f_2(x)= \displaystyle\sum_{i=1}^{q-1} \ln\Big(\frac{x}{i}+1\Big)$.

Assume $x\in\RR_{\geq 1}$. Since $f'_{1}(x)=\ln q$ and  $f'_{2}(x)=\displaystyle\sum_{i=1}^{q-1} \frac{1}{x+i}$, we obtain 
\begin{eq}\label{eq_der}
f_{1}'(x)>f_{2}'(x), 
\end{eq}%
where we use inequality $f_{2}'(1)\geq f_{2}'(x)$ and the well-known upper bound on a partial sum $f_{2}'(1)$ of the harmonic series:
$$f_{2}'(1)=\frac{1}{2} + \cdots + \frac{1}{q} < \ln q. $$%
Functions $f_1(x)$ and $f_2(x)$ are strictly increasing over $\RR_{\geq 1}$ and $f_1(1)<f_2(1)$. We claim that 
\begin{eq}\label{eq_claim_f1f2}
f_1(a)> f_2(a) \;\text{ for some } a>1.
\end{eq}
To prove the claim, we consider the following three cases.
\begin{enumerate}
\item[$\bullet$] If $q=2$, then $f_2(x)=\ln(x+1)$ and $f_1(4)>f_2(4)$. 

\item[$\bullet$] If $q=3$, then $f_2(x)=\ln(1+x)+\ln(1+x/2)$ and $f_1(4)>f_2(4)$.

\item[$\bullet$] Assume $q>3$. Then $f_1(q) = \ln q^2 + (q-3) \ln q$ and \[f_2(q)=\ln \frac{(q+1)(q+2)}{2} +  \displaystyle\sum_{i=3}^{q-1} \ln\Big(\frac{q}{i}+1\Big).\] Since $q^2 > (q+1)(q+2)/2$ and $q>(q+i) /i$ for $i\geq3$, we obtain that $f_1(q)>f_2(q)$.
\end{enumerate}

Claim~(\ref{eq_claim_f1f2}) together with inequality $f_1(1)<f_2(1)$ implies that $f_1(x_0)=f_2(x_0)$ for some $x_0\in \RR_{\geq 1}$ with $1<x_0<a$. Inequality~(\ref{eq_der}) implies that $x_0=x_0(q)$ is the unique solution of the equation  $f_1(x)=f_2(x)$ over  $\RR_{\geq 1}$. Moreover, we can see that for $x\in \RR_{\geq 1}$ we have $f_1(x)<f_2(x)$ if and only if $x<x_0$.  Obviously, $x_0$ is also the unique solution of the equation
$$q^{x-1} = (x+1)\Big(\frac{x}{2}+1\Big)\cdot \ldots \cdot \Big(\frac{x}{q-1}+1\Big)$$
over $\RR_{\geq 1}$. 

In case $q>3$ we have $f_1(q)>f_2(q)$ and we may take $a=q;$  hence $x_0<q$.  The requirements is proven.
\end{proof}

Let us remark that the following lemma which is an easy corollary of~\cite[Theorem 1.1]{Kemper_Lopatin_Reimers_2022} describes when $S(n)$ is  a minimal separating set having the least possible number of elements, but for our purposes, we need more explicit condition on $n$. 

\begin{lemma}\label{lemma_referee}
The set $S(n)$ is a minimal separating set for $\FF_q[V]^{\Sym_n}$ having the least possible number of elements if and only if
$$q^{n-1}<\binom{n+q-1}{n}.$$
\end{lemma}
\begin{proof} It follows from equivalence~(\ref{eq_new11}).
\end{proof}

Given $x_0=x_0(q)$ from the formulation of Theorem~\ref{theo_min_sep}, define $\chi_q\in \NN$ as follows:
\begin{eq}\label{eq_zeta}
\chi_q=
\left\{
\begin{array}{cc}
x_0-1, & \text{ if } x_0\in \NN \\
\lfloor x_0\rfloor, & \text{ if }x_0\not\in\NN \\
\end{array}
\right.
\end{eq}%
Note that $\chi_q$ is defined for an arbitrary integer $q\geq2$, not only for the power of a prime. Theorem~\ref{theo_min_sep} implies the following corollary.

\begin{corollary}\label{corArt}
The set $S(n)$ is a minimal separating set for $\FF_q[V]^{\Sym_n}$ having the least possible number of elements if and only if $n\leq \chi_q$. Moreover, 
$$1\leq \chi_q < q \text{ in case }q>3.$$
\end{corollary}

\begin{definition}\label{rem1}
By straightforward calculations, using a computer, we can see that
\begin{enumerate}
\item[$\bullet$] $\chi_2 = 2;$

\item[$\bullet$] $\chi_q = 3$ for $3\leq q \leq 17;$

\item[$\bullet$] $\chi_q = 4$ for $18\leq q \leq 109;$

\item[$\bullet$] $\chi_q = 5$ for $110 \leq q \leq 704;$

\item[$\bullet$] $\chi_q = 6$ for $705 \leq q \leq 5018;$

\item[$\bullet$] $\chi_q = 7$ for $5019\leq q \leq 10^4$.
\end{enumerate}
\end{definition}

%

To prove a lower bound on~$\chi_q$ from Theorem~\ref{theo2} (see below) we need the following technical lemma.

\begin{lemma}\label{lemma_tecnical} For every $q \geq e^{e^{2}}$ we have 
$$ \ln q  - (2\ln(\ln q ) + 1) \ln(\ln(\ln q)) > 0.$$
\end{lemma}
\begin{proof}
Assume $x\in\RR_{\geq e}$. Then for 
$$
h(x)= x - (2\ln x+1) \ln(\ln x).
$$
we have 
$$h'(x)=\frac{w(x)}{x\ln x}, \;\text{ where }\; w(x)=x\ln x-2\ln x-2\ln(\ln x) \ln x - 1.
$$
Since 
$$w'(x)=\frac{(x\ln x-2\ln(\ln x)) + (x - 4)}{x}>0,
$$ 
for all $x\ge e^2$ and $w(e^2)=2 e^2 - 4 \ln 2 - 5>0$, we obtain that $w(x)>0$ for all $x\ge e^2$. Therefore, $h'(x)>0$ for all $x\ge e^2$. Hence, the inequality $h(e^2)=e^2 - 5\ln 2>0$ implies that $h(x)>0$ for all $x\ge e^2$. In particular, $h(\ln q)>0$ for all $q\ge e^{e^2}$. The required statement is proven.
\end{proof}

\begin{theorem}\label{theo2}
    We have $\chi_{q} \ge \lfloor \ln (\ln q) \rfloor$.
\end{theorem}
\begin{proof}
If $q < e^{e^2}$, then $\lfloor \ln (\ln q) \rfloor \le 1 \leq\chi_{q},$
and the required statement is proven.

Assume that $q \geq e^{e^2}$. Define \[f(x) = q^{x-1}, \quad g(x) = \frac{(x+1)\cdot \ldots \cdot (x+q-1)}{(q-1)!}\] for $x\in\RR_{\geq1}$.  Recall that $x_0=x_0(q)$ from definition~(\ref{eq_zeta}) of $\chi_{q}$ is the unique solution of the equation $f(x)=g(x)$ over $\RR_{\geq1}$ and $1=f(1)<g(1)=q$. Hence, to prove the theorem is sufficient to show that
\begin{eq}\label{eq_fg}
f(b) < g(b)
\end{eq}%
for $b =\lfloor \ln (\ln q )\rfloor\geq2$, since inequality~(\ref{eq_fg}) implies that $b<x_0$. Inequality~(\ref{eq_fg}) is equivalent to the inequality $\ln (f(b)) < \ln (g(b))$. 

For short, define $a = b + q - 1\geq q+1$. Then
$$
g(b) = \binom{a}{q-1} \;\;\text{ and }\;\;\; \ln g(b) = \ln a! - \ln b! - \ln (q-1)!
$$%
Using well-known inequalities
$$
\sqrt{2 \pi k} \Big( \frac{k}{e} \Big)^{k} < k! < 2 \sqrt{\pi k} \Big( \frac{k}{e} \Big)^{k} \quad \text{ for all }\; k \geq 1,
$$
we obtain that 
\begin{eq}\label{eq_gb}
\ln g(b) > a\ln(a) - \biggl(q-\frac{1}{2}\biggl)\ln(q-1) 
+\frac{1}{2}\biggl(\ln(a) - 2b\ln(b) - \ln(b)\biggl) -\frac{1}{2}\ln(8\pi).
\end{eq}

By the definition of $b$, we have $2\leq b \le \ln(\ln q)$. Therefore, 
$$
\ln(a) - 2b\ln(b) - \ln(b) \geq \ln(q) - (2\ln(\ln q) + 1) \ln(\ln(\ln q))>0
$$
by Lemma~\ref{lemma_tecnical}. Thus inequality~(\ref{eq_gb}) implies that 
$$
\ln{g(b)}> a\ln(a)-\biggl(q-\frac{1}{2}\biggl)\ln(q-1) -\frac{1}{2}\ln(8\pi).
$$
Applying inequality $a\geq q+1$, we obtain
\[
\begin{aligned}
    \ln g(b) &> (q + b - 1)\ln(q + 1) 
    - \biggl(q - \frac{1}{2}\biggr)\ln(q - 1) 
    - \frac{1}{2}\ln(8\pi)= \\
    &= (b - 1)\ln(q + 1) 
    + \biggl(q - \frac{1}{2}\biggr)\biggl(\ln(q + 1) - \ln(q - 1)\biggr)+ \\
    &\quad + \frac{1}{2}\biggl(\ln(q + 1) - \ln(8\pi)\biggr) > (b - 1)\ln q = \ln f(b).
\end{aligned}
\]

The required statement is proven.
\end{proof}

\begin{corollary}\label{cor2}
We have $\lim_{q\to \infty}\chi_q = +\infty$.
\end{corollary}

\begin{corollary}\label{cor3} For every $n\geq2$ there exists $q$ such that $S(n)$ is a minimal separating set for $\FF_q[V]^{\Sym_n}$ having the least possible number of elements.
\end{corollary}

\subsection*{Acknowledgments} The first author was supported by FAPESP 2021/01690-7. We are grateful for this support.




\EditInfo{October 26, 2024}{January 2, 2025}{Ivan Kaygorodov}

\begin{thebibliography}{10}

\bibitem{Aberth_1964}
O.~Aberth.
\newblock The elementary functions in a finite field of prime order.
\newblock {\em Illinois J. Math.}, 8:132--138, 1964.

\bibitem{Cavalcante_Lopatin_2020}
F.~Cavalcante and A.~Lopatin.
\newblock Separating invariants of three nilpotent $3\times3$ matrices.
\newblock {\em Linear Algebra Appl.}, 607:9--28, 2020.

\bibitem{Domokos_2020_add}
M.~Domokos.
\newblock Addendum to ``{C}haracteristic free description of semi-invariants of
  $2\times2$ matrices''[{J}. {P}ure {A}ppl. {A}lgebra 224 (2020), no. 5,
  106220].
\newblock {\em J. Pure Appl. Algebra}, 224(6):106270, 2020.

\bibitem{Domokos_2020}
M.~Domokos.
\newblock Characteristic free description of semi-invariants of $2\times2$
  matrices.
\newblock {\em J. Pure Appl. Algebra}, 224(5):106220, 2020.

\bibitem{Domokos_Miklosi_2023}
M.~Domokos and B.~Mikl${\rm \acute{o}}$s.
\newblock Symmetric polynomials over finite fields.
\newblock {\em Finite Fields Appl.}, 89:102224, 2023.

\bibitem{Lopatin_Ferreira_2023}
R.~Ferreira and A.~Lopatin.
\newblock Minimal generating and separating sets for ${O}(3)$-invariants of
  several matrices.
\newblock {\em Operators and Matrices}, 17(3):639--651, 2023.

\bibitem{Kaygorodov_Lopatin_Popov_2018}
I.~Kaygorodov, A.~Lopatin, and Y.~Popov.
\newblock Separating invariants for $2\times 2$ matrices.
\newblock {\em Linear Algebra Appl.}, 559:114--124, 2018.

\bibitem{Kemper_Lopatin_Reimers_2022}
G.~Kemper, A.~Lopatin, and F.~Reimers.
\newblock Separating invariants over finite fields.
\newblock {\em J. Pure Appl. Algebra}, 226:106904, 2022.

\bibitem{Lopatin_2024_IJAC}
A.~Lopatin.
\newblock On $m$-tuples of nilpotent $2\times2$ matrices over an arbitrary
  field.
\newblock {\em Int. J. Algebra Comput.}, 34(8):1253--1272, 2024.

\bibitem{Lopatin_Martins_2023}
A.~Lopatin and P.~Muniz~Martins.
\newblock Separating invariants for two-dimensional orthogonal groups over
  finite fields.
\newblock {\em Linear Algebra Appl.}, 692:71--83, 2024.

\bibitem{Lopatin_Reimers_2021}
A.~Lopatin and F.~Reimers.
\newblock Separating invariants for multisymmetric polynomials.
\newblock {\em Proc. Amer. Math. Soc.}, 149:497--508, 2021.

\bibitem{Lopatin_Zubkov_ANT}
A.~Lopatin and A.~Zubkov.
\newblock Separating ${G}_2$-invariants of several octonions.
\newblock {\em Algebra Number Theory}, 18(12):2157--2177, 2024.

\bibitem{Reimers_2020}
F.~Reimers.
\newblock Separating invariants for two copies of the natural ${S}_n$-action.
\newblock {\em Commun. Algebra}, 48:1584--1590, 2020.

\end{thebibliography}
\end{document}